\documentclass[12pt]{amsart}
\usepackage[applemac]{inputenc}
\usepackage{bm}

\usepackage[dvipsnames,svgnames,x11names,hyperref]{xcolor}
\usepackage[hyphens]{url}
\usepackage[colorlinks=true,linkcolor=Maroon,citecolor=blue,urlcolor=blue,hypertexnames=false,linktocpage]{hyperref}
\usepackage{bookmark}
\usepackage{amsmath,thmtools,mathtools}
\mathtoolsset{showonlyrefs=true}

\usepackage{blindtext}

\newcounter{mgncount}

\usepackage{fancyhdr}
\usepackage{esint}
\usepackage{enumerate}

\usepackage{pictexwd,dcpic}
\usepackage{graphicx}
\usepackage{caption}
\usepackage{esint}
\usepackage{graphicx}

\newtheorem*{thm-}{Theorem}
\declaretheorem[name=Theorem,numberwithin=section]{thm}
\declaretheorem[name=Remark,style=remark,sibling=thm]{rem}
\declaretheorem[name=Lemma,sibling=thm]{lemma}
\declaretheorem[name=Proposition,sibling=thm]{prop}

\declaretheorem[name=Corollary,sibling=thm]{cor}

\declaretheorem[name=Theorem,numbered=no]{theorem}

\numberwithin{equation}{section}

\newcommand{\ti}{\tilde}

\newcommand{\ov}{\bar}

\newcommand{\bbR}{\mathbb{R}}

\newcommand{\bbS}{\mathbb{S}}

\newcommand{\Sn}{\mathbb{S}^n}

\newcommand{\al}{\alpha}

\newcommand{\de}{\delta}

\newcommand{\De}{\Delta}

\newcommand{\cH}{\mathcal{H}}

\newcommand{\cK}{\mathcal{K}}

\renewcommand{\(}{\left(}
\renewcommand{\)}{\right)}

\newcommand{\pf}[1]{\begin{proof}#1 \end{proof}}
\newcommand{\eq}[1]{\begin{equation}\begin{alignedat}{2} #1 \end{alignedat}\end{equation}}

\newcommand{\ra}{\rightarrow}

\newcommand{\q}{\quad}

\makeatletter
\@namedef{subjclassname@2020}{%
	\textup{2020} Mathematics Subject Classification}
\makeatother

\begin{document}
	\title[Cone-volume measure with near constant density]
	{Stability of the cone-volume measure with near constant density}
	\author[Y. Hu, M. N. Ivaki]{Yingxiang Hu, Mohammad N. Ivaki}
	
\begin{abstract}
We prove that if the density of the cone-volume measure of a smooth, strictly convex body with respect to the spherical Lebesgue measure is nearly constant, then a homothetic copy of the body is close to the unit ball in the $L^2$-distance.
\end{abstract}
\maketitle

\section{Introduction}
The following theorem in the origin-symmetric case was proved by Firey \cite{Fir74}, and without the origin-symmetry assumption is due to Gage for $n=1$ \cite{Gag84}, Andrews for $n=2$ \cite{And99}, and Choi-Daskalopoulos for $n\geq 3$ \cite{CD16,BCD17}.

\begin{theorem}
Let $K$ be a smooth, strictly convex body with the support function $h$ and Gauss curvature $\mathcal{K}$. 
If $\mathcal{K}=h$, then $K$ is the unit ball.
\end{theorem}

We prove a stability version of this theorem here. Write $B$ for the unit ball of the Euclidean space $\bbR^{n+1}$.  Throughout this paper, $\int$ refers to integration on the unit sphere $\bbS^n$.  

\begin{thm}\label{s1:thm-stability}
Let $K$ be a smooth, strictly convex body with the support function $h>0$.
Then 
\eq{
\de_2\(\bar{K},B\) \leq \gamma \left(\frac{\max \frac{h}{\cK}}{\min \frac{h}{\cK}}-1\right)^\frac{1}{2},
}
where $\gamma$ depends only on $n$,
\eq{\label{s1:translated-rescaled-body}
\bar{K}=\frac{K-c(K)}{\int  h d\theta/\int d\theta},\quad c(K)=\frac{\int Dh dV}{\int dV},
}
and $\frac{1}{n+1}V$ is the cone-volume measure of $K$.
\end{thm}
 Note that \autoref{s1:thm-stability} does not require any assumption of smallness for the quantity 
\[
\frac{\max \frac{h}{\cK}}{\min \frac{h}{\cK}}-1.
\]
 When $K$ is assumed to be origin-symmetric, a similar stability result was proved in \cite{Iva22}. See also \cite{BD2} for a related stability result.

To prove \autoref{s1:thm-stability}, we employ an inequality from \cite{IM23a}, where a short proof of uniqueness for $\cK = h$ was given. However, the final steps of that proof, relying on the Poincar\'{e} inequality on the sphere, are not suitable for our purpose. The refined approach presented here has the added advantage of also allowing us to establish the uniqueness of solutions to $\cK = h^{1-p}$ for the whole range $-n-1 < p \leq 0$. The interval $p \in (-1,0)$ was previously absent in the argument of \cite{IM23a}; see \autoref{Appendix}.

A quick corollary of \autoref{s1:thm-stability} is the uniqueness of solutions to the regular logarithmic Minkowski problem without a symmetry condition, provided the prescribed data is sufficiently close to $1$ in the $C^{\alpha}$-norm.

\begin{cor}\cite{CFL22,BS23}\label{main cor} Let $\alpha\in (0,1)$, $n\geq 2$ and $f\in C^{\alpha}(\bbS^n)$. Then there exists a constant $\varepsilon_0>0$ depending only on $n,\, \alpha$, such that if $\|f-1\|_{C^\al} \leq \varepsilon$ for some $\varepsilon\in (0,\varepsilon_0)$, then the log-Minkowski problem $h/\cK=f$ has a unique, positive, strictly convex solution.
\end{cor}
This corollary was recently proved by Chen, Feng, and Liu \cite{CFL22} for $n=2$, and by B\"or\"oczky and Saroglou for the case $n \geq 2$ and $h^{1-p}/\cK = f$ for $0 \leq p < 1$ in \cite{BS23}. We refer the reader to \cite{Bor22,IM23b,Mil23,KM22} and the references therein for the importance of uniqueness results on the log-Minkowski problem.

\section{Background}
Let $(\mathbb R^{n+1},\delta:=\langle\, ,\rangle,D)$ denote the Euclidean space with its standard inner product and flat connection, and let $(\mathbb S^n,\bar{g},\bar{\nabla})$ denote the unit sphere equipped with its standard round metric and Levi-Civita connection. Moreover, for $f\in C^2(\bbS^n)$, we define  $\bar \Delta f=\operatorname{div}_{\bar g}(\bar \nabla f)$.

A compact, convex set with non-empty interior is called a convex body. Let $L$ be a convex body. The support function of $L$ is defined as	
	\begin{align*}
		h_L(x):=\max\{\langle x,y\rangle:~y\in L\}, \quad x\in \mathbb S^n.
	\end{align*}
The $L^2$-distance of two convex bodies $L_1,L_2$ is defined by
\eq{
\de_2(L_1,L_2):=\left(\frac{1}{\int d\theta}\int |h_{L_1}-h_{L_2}|^2d\theta\right)^{\frac{1}{2}},
}
and their Hausdorff distance is defined as
\eq{
\de_{H} (L_1,L_2):=\max_{\bbS^n} |h_{L_1}-h_{L_2}|.
}

Let $K$ be a smooth, strictly convex body in $\mathbb R^{n+1}$ with the origin in its interior. Write $\mathcal{M}=\partial K$ for the boundary of $K$. The Gauss map of $\mathcal{M}$, denoted by $\nu_K$, takes the point $p\in \mathcal{M}$ to its unique unit outward normal $x=\nu_K(p)\in \mathbb S^n$.
The inverse Gauss map $X=\nu_K^{-1}:\mathbb S^n\ra \mathcal{M}$ is given by
	\begin{align*}
		X(x)=Dh_K(x)=\bar{\nabla} h_K(x)+h_K(x)x, \quad x\in \mathbb S^{n}.
	\end{align*}
The support function of $K$ can also be expressed as
	\begin{align*}
		h_K(x)=\langle X(x),x\rangle=\langle \nu_K^{-1}(x),x\rangle, \quad x\in \mathbb S^n.
	\end{align*}
The Gauss curvature of $K$ (or $\mathcal{M}$) viewed as a function on the unit sphere is defined by
	\begin{align*}
		\frac{1}{\mathcal{K}_K(x)}:=\left.\frac{\det(\bar{\nabla}^2 h_K+\bar{g}h_K)}{\det(\bar{g})}\right|_x, \quad x\in \mathbb S^n.
	\end{align*}
	
Moreover, define the measure $dV_K=(h_K/\cK_K) d\theta$, where $\theta$ is the spherical Lebesgue measure of $\mathbb{S}^n$. The measure  $\frac{1}{n+1}V_K$ is called the cone-volume measure of $K$. From now on, when we work with the convex body $K$, it is convenient to drop the index $K$: for example, $h=h_K$, $\cK=\cK_K$, $V=V_K$.

\section{Stability}
We recall the following inequality in \cite[Lem. 3.2]{IM23a}.
\begin{lemma}
Let $X=Dh:\mathbb S^n \ra \partial K$. Then we have
\eq{ \label{s1:key-ineq}
n\int |X|^2 dV\leq \int h(\ov \De h+nh)dV+n\frac{|\int X dV|^2}{\int dV}.
}
\end{lemma}

Recall that $c(K)=\frac{\int X dV}{\int dV}$. Let us put  $\tilde{K}:=K-c(K)$ and $\tilde{h}=h_{\tilde{K}}$.
\begin{prop}\label{prop: basic est}
Let $m\leq \frac{h}{\cK}\leq M$. Then we have
\eq{
n \int |D\ti h|^2 d\theta \leq  \frac{M}{m} \int \ti h(\ov\De \ti h+n \ti h) d\theta.
}
\end{prop}
\begin{proof}
Inequality \eqref{s1:key-ineq} can be rewritten as
\eq{
n\int |Dh-c(K)|^2 dV \leq \int h(\ov\De h+nh)dV.
}
Using $\ov\De x+nx=0$, we have
\eq{
\ov\Delta \ti h(x)+n\ti h(x)=\ov\Delta h(x)+n h(x).
}
Therefore,
\eq{
n\int |D\ti h|^2 dV \leq \int h(\ov\De \ti h+n\ti h)dV
}
and
\eq{
n \int |D\ti h|^2 d\theta &\leq \frac{M}{m} \int (\ov\De \ti h+n \ti h) h d\theta=\frac{M}{m} \int (\ov\De \ti h+n \ti h) \ti h d\theta.
}
\end{proof}

\pf{[Proof of \autoref{s1:thm-stability}]
Let 
\eq{
M=\max \frac{h}{\cK},\q m=\min \frac{h}{\cK},\q \varepsilon:=\frac{M}{m}-1.
} 
In view of  \autoref{prop: basic est},
\eq{ \label{s1:stability-estimate}
\(n+1+\varepsilon\)\int |\ov \nabla \ti h|^2 d\theta \leq n\varepsilon \int \ti h^2 d\theta.
}
Applying the Poincar\'e inequality to $\ti h$, we have
\eq{ \label{s1:Poincare-ineq}
n\int \(\ti h- \fint \ti h d\theta\)^2 d\theta \leq \int |\ov\nabla \ti h|^2 d\theta,
}
where 
\eq{
\fint \varphi d\theta:=\frac{1}{\int d\theta}\int \varphi d\theta, \quad \forall \varphi\in C(\bbS^n). 
}
Combining \eqref{s1:stability-estimate} with \eqref{s1:Poincare-ineq} yields
\eq{
\fint \(\ti h- \fint \ti h d\theta\)^2 d\theta \leq \frac{\varepsilon}{n+1}\fint \ti h^2 d\theta.
}
It follows that
\eq{
\fint \(\frac{\ti h}{\fint \ti h d\theta}-1\)^2 d\theta \leq \frac{\varepsilon }{n+1}\frac{\fint \ti h^2 d\theta}{\(\fint \ti h d\theta\)^2}.  
}
Next, by a standard argument, we show that the right-hand side is bounded. Note that
 for $v\in \bbS^n$, we have
\eq{
\langle c(K),v\rangle\leq \max_{x\in \bbS^n} \langle Dh(x),v\rangle=\max_{y\in K} \langle y,v\rangle=h(v).
}
Hence, $\tilde{h}\geq 0$.
Let $M_{\ti h}=\max \ti h$. There exists a unit vector $w\in \bbS^n$ such that 
$M_{\ti h}=\ti h(w)$. Due to convexity, for any $x\in \bbS^n$, we have
\eq{
\ti h(x)=\max_{y\in \ti K}\langle x,y\rangle \geq \langle x,w\rangle M_{\ti h}.
}
Therefore, for some $c_1$, depending on $n$, we have
\eq{ \label{s1:lower-bound-width}
\fint \ti h d\theta \geq \frac{1}{\int d\theta}\int_{\langle x,w\rangle \geq \frac{1}{2}} \ti h d\theta \geq \frac{M_{\ti h}}{2\int d\theta} \int_{\langle x,w\rangle \geq \frac{1}{2}} d\theta \geq c_1 M_{\ti h}
} 
and
\eq{
\(\fint \ti h^2 d\theta\)^\frac{1}{2} \leq M_{\ti h}  \leq  \frac{1}{c_1}\fint \ti h d\theta.
}
Hence,
\eq{\label{s1:quantitative-estimate}
\fint \(\frac{\ti h}{\fint \ti h d\theta}-1\)^2 d\theta \leq \frac{\varepsilon}{(n+1)c_1^2}.
}
}

 \begin{rem}A simple example of a non-origin centred ball shows that the exponent $1/2$ in \autoref{s1:thm-stability} cannot be replaced by an exponent smaller than one. Let $K=B+w$ be the unit ball centred at the point $w$ with $|w|<1$. Then $h(x)=1+\langle w,x\rangle$ and $Dh(x)=x+w$ for $x\in \mathbb{S}^n$. Moreover, we have
\eq{
c(K)=\frac{\int Dh dV}{\int dV}& =\frac{\int (x+w) (1+\langle w,x\rangle) d\theta(x)}{\int 1+\langle w,x\rangle d\theta(x)}\\
&=w+\frac{\int  x\langle w,x\rangle d\theta(x)}{\int d\theta}\\
&=\frac{n+2}{n+1}w,
}
where we used the identity
\eq{
\int x\otimes x d\theta(x)=\frac{1}{n+1}\int d\theta \operatorname{Id}.
}
Therefore, we have
\eq{
K-c(K)&=B+w-\frac{n+2}{n+1}w=B-\frac{1}{n+1}w,\\
\de_2\(\bar{K},B\) &= \(\frac{1}{\int d\theta} \int \frac{\langle w,x\rangle^2}{(n+1)^2} d\theta(x)\)^\frac{1}{2}=\frac{|w|}{(n+1)^\frac{3}{2}},\\
\frac{\max \frac{h}{\cK}}{\min \frac{h}{\cK}}-1 &=\frac{1+|w|}{1-|w|}-1= \frac{2|w|}{1-|w|}.
}
\end{rem}

As an application of \autoref{s1:thm-stability}, we obtain a uniform diameter bound when the density $f$ of the cone-volume measure with respect to the spherical Lebesgue measure is close to $1$. This diameter bound is the main ingredient in the proof of \autoref{main cor}.
\begin{lemma}\cite[Lem. 7.6.4]{Sch14}
Let $K_1$, $K_2$ be two convex bodies in $\bbR^{n+1}$. Then there holds
\eq{\label{s1:L2-to-Linfty}
\de_2(K_1,K_2)^2 \geq \alpha_n \operatorname{diam}(K_1\cup K_2)^{-n} \de_H(K_1,K_2)^{n+2},
}
where $\al_n$ is a dimensional constant and $\operatorname{diam}(K_1\cup K_2)$ is the diameter of the set $K_1\cup K_2$.
\end{lemma}

\begin{prop}\label{s1:thm-upper-bound-support}
    There exist $\varepsilon_0>0$ and $C=C(\varepsilon_0,n)$  with following property. If $K$ is a smooth, strictly convex body containing the origin in its interior, such that
    \eq{ \label{s1:assump-f}
    1-\varepsilon \leq \frac{h}{\cK}\leq 1+\varepsilon
    }
for some $\varepsilon\in (0,\varepsilon_0)$,  then $h \leq C$.
\end{prop}

\begin{proof}
In view of \eqref{s1:lower-bound-width}, the support function of $\bar{K}$ satisfies
\eq{
h_{\bar{K}}(x)=\frac{\ti h(x)}{\fint \ti h d\theta} \leq \frac{1}{c_1}, \quad \forall x\in \bbS^n.
}
Then we have
\eq{
\operatorname{diam}(\bar{K} \cup B_1)\leq 2\left(1+\frac{1}{c_1}\right).
}
On the other hand, by \autoref{s1:thm-stability} we have
\eq{
\de_2(\bar{K},B_1) \leq \gamma \varepsilon_0^\frac{1}{2}.
}
Hence, from \eqref{s1:L2-to-Linfty} it follows for some constant $c_2$, depending only on $n$,
\eq{
\de_H(\bar{K},B_1) &\leq \al_n^{-\frac{1}{n+2}}\operatorname{diam}(\bar{K} \cup B_1)^\frac{n}{n+2} \de_2(\bar{K},B_1)^\frac{2}{n+2}\leq c_2 \varepsilon_0^\frac{1}{n+2}.
}
Thus, we have
\eq{\label{s1:Hausdorff-estimate}
1-c_2 \varepsilon_0^{\frac{1}{n+2}} \leq \frac{\ti h}{\fint \ti h d\theta} \leq 1+c_2 \varepsilon_0^{\frac{1}{n+2}},
}
and for $\varepsilon_0$ with $c_2 \varepsilon_0^{\frac{1}{n+2}}< 1$,
\eq{\label{s1:pinching-translated-h}
\frac{\max \ti h}{\min \ti h} \leq \frac{1+c_2 \varepsilon_0^{\frac{1}{n+2}}}{1-c_2 \varepsilon_0^{\frac{1}{n+2}}}.
}
Using \eqref{s1:assump-f}, we have
\eq{
(\max\ti h)^{n+1} \geq \frac{V(\tilde{K})}{V(B)}=\frac{V(K)}{V(B)}\geq 1-\varepsilon.
}
Hence, there exists a constant $c_3>0$ depending on $n$ such that
\eq{
\max\ti h \geq c_3.
} 
Substituting this into \eqref{s1:pinching-translated-h} and assuming $c_2 \varepsilon_0^{\frac{1}{n+2}}< \frac{1}{2}$, we obtain
\eq{
\min \ti h \geq \frac{c_3}{3}.
}
This means that an origin-centred ball of radius $c_3/3$ is contained in $K-c(K)$, and hence the inradius of $K$ is at least  $c_3/3$. Due to \eqref{s1:assump-f}, $V(K)\leq (1+\varepsilon)V(B)$. Therefore, there exists $C$, depending only on $\varepsilon_0,\,n$, such that $h<C$.
\end{proof}

\pf{[Proof of \autoref{main cor}]
Follows from \autoref{s1:thm-upper-bound-support}; see, for example, \cite{CFL22} for details.
}

\section{Uniqueness}\label{Appendix}

\begin{thm}\label{s2:main-thm}\cite{Gag84,And03,And99,AC12,BCD17}
Let $K$ be a smooth, strictly convex body. If $\cK=h^{1-p}$ with $p\in (-n-1,1)$, then $K$ is the unit ball.  
\end{thm}

In \cite{IM23a}, using the local Brunn-Minkowski inequality, a new proof of this theorem was given for the cases $-n-1 < p \leq -1$ and $p = 0$. Here, we present an argument that can also handle the case $p \in (-1,0)$. The new ingredient is the following integral identity.
\begin{lemma} \label{s2:lem-isotropic}
Let $p\neq -(n+1)$. If $dV=h^pd\theta$, then 
\begin{align*}
\int x\otimes x dV=\left(\frac{1}{n+1}\int dV\right)\mathrm{Id}.
\end{align*}
\end{lemma}
\begin{proof}
For $p=0$ the identity is trivial. We may assume $p\neq 0$. Let us take $w_1,w_2\in \bbR^{n+1}$ and define $\ell_i(x):=\langle x,w_i\rangle:\Sn\to \mathbb{R}$. By the divergence theorem, we have
\begin{align*}
\int \langle w_1,x\rangle \langle X(x),w_2\rangle \frac{1}{\cK}d\theta(x)
&=	\int_{u\in\partial K}\langle w_1,\nu(u)\rangle \langle u,w_2\rangle \cH^n(du)\\
&=\int_{u\in K}\operatorname{div}_{\bbR^{n+1}}( \langle u,w_2\rangle w_1)du\\
&=\frac{\int dV}{n+1}\langle w_1,w_2\rangle.
\end{align*}
Since $h^{p-1}=\frac{1}{\cK}$, we obtain
\begin{align}\label{div x x polar}
\int \langle w_1,x\rangle\langle X(x), w_2\rangle h^{p-1}d\theta(x)
&=\frac{\int dV}{n+1}\langle w_1,w_2\rangle. 
\end{align}
Moreover, by integrating by parts, we have
\begin{align*}
\int \ell_1\ell_2 \bar{\Delta}h^pd\theta=-p\int \ell_1\langle \bar{\nabla}\ell_2,\bar{\nabla}h\rangle h^{p-1}d\theta-p\int \ell_2\langle \bar{\nabla}\ell_1,\bar{\nabla}h\rangle h^{p-1}d\theta.
\end{align*}
The left-hand side of this identity can be rewritten as
\begin{align*}
\int \ell_1\ell_2 \bar{\Delta}h^pd\theta&=\int  \bar{\Delta}(\ell_1\ell_2) dV\\
&=-2(n+1)\int \ell_1\ell_2dV+2\int dV \langle w_1,w_2\rangle.
\end{align*}
On the other hand, using \eqref{div x x polar}, the first term on the right-hand side  (similarly, the second term) simplifies to
\begin{align*}
-p\int \ell_1\langle \bar{\nabla}\ell_2,\bar{\nabla}h\rangle h^{p-1}d\theta&=-p\int \ell_1\langle w_2,X-hx\rangle h^{p-1}d\theta\\
&=p\int \ell_1\ell_2dV-\frac{p}{n+1}\int  dV\langle w_1,w_2\rangle.
\end{align*}
Putting everything together, we find
\begin{align*}
&(n+1+p)\left(\int \ell_1\ell_2dV-\frac{\int dV}{n+1}\langle w_1,w_2\rangle\right)=0.
\end{align*}
Hence, for $p\neq -(n+1)$ there holds
\[
\int \ell_1\ell_2dV=\frac{\int dV}{n+1}\langle w_1,w_2\rangle\quad \forall w_1,w_2\in \Sn.
\]
\end{proof}

\begin{lemma}\label{some lemma}If $dV=h^pd\theta$ and $p\neq -n-1$, then
\eq{
\int |\ov\nabla h|^2 dV = \int |\ov\nabla \ti h|^2 dV + \frac{n(n+1-p)|c|^2}{(n+1)(n+1+p)}\int dV.  
}
\end{lemma}
\begin{proof}
Let $c=c(K)$. Recall that $\tilde{h}(x)=h(x)-\langle x, c\rangle$. We have
\eq{
    \int |\ov\nabla h|^2- |\ov\nabla \ti h|^2dV &=\int \langle \ov\nabla (h-\ti h),\ov\nabla (h+\ti h)\rangle dV\\
    &=\int \langle \ov\nabla \langle c,x\rangle,2\ov\nabla h-\ov\nabla\langle c,x\rangle\rangle dV \\
    &=\int \langle \ov\nabla \langle c,x\rangle,2X-\ov\nabla\langle c,x\rangle\rangle dV \\
    &=\int \langle c-\langle c,x\rangle x,2X-c+\langle c,x\rangle x\rangle dV\\
    &=\int \langle c-\langle c,x\rangle x,2X-c\rangle dV\\
    &=\int |c|^2-2\langle c,x\rangle h +\langle c,x\rangle^2 dV.  \label{last step}
}
Here, we used $c-\langle c,x\rangle x=\ov\nabla \langle c,x\rangle\in T\bbS^n$ and $\int X dV=(\int dV) c $. 
Using $\ov\Delta x+nx=0$ and by integrating by parts, we have
\eq{
\int hx dV &=-\frac{1}{n}\int h^{p+1}\ov\De x d\theta=\frac{p+1}{n}\int h^p\ov\nabla h d\theta,
}
and 
\eq{
\int X dV&=\int (\ov\nabla h +hx)h^p d\theta 
         =\frac{n+p+1}{n}\int h^p \ov\nabla h d\theta.
}
Therefore,
\eq{ \label{s2:key-identity-I}
\int h x dV&=\frac{p+1}{n+p+1}\int XdV
                        =\frac{p+1}{n+p+1}\(\int dV\)c.
}
Moreover, by \autoref{s2:lem-isotropic}, for any $p\neq -(n+1)$,
\eq{ \label{s2:key-identity-II}
\int \langle c,x\rangle^2 dV=\int x\otimes x\Big|_{(c,c)}dV =\frac{\int dV}{n+1}|c|^2.
}
Now, substituting  \eqref{s2:key-identity-I} and $\eqref{s2:key-identity-II}$ into \eqref{last step}, we finally obtain 
\eq{
\int |\ov\nabla h|^2-  |\ov\nabla \ti h|^2dV &=\(1-\frac{2(p+1)}{n+p+1}+\frac{1}{n+1}\)|c|^2\int dV\\
&=\frac{n(n+1-p)|c|^2}{(n+1)(n+1+p)} \int dV.
}
\end{proof}

\begin{proof}[Proof of \autoref{s2:main-thm} \mbox{when $p\in (-n-1,0]$}] 
By \cite[(4.4)]{IM23a}, we have
\eq{
\int |\ov\nabla h|^2 dV \leq \frac{n|c|^2}{n+1+p}\int dV. 
}
Due to  \autoref{some lemma},
\eq{
\int|\ov\nabla \ti h|^2 dV &=\int |\ov\nabla h|^2dV-\frac{n(n+1-p)|c|^2}{(n+1)(n+1+p)} \int dV\\
&\leq \frac{np|c|^2}{(n+1)(n+1+p)} \int dV.
}
Therefore, for $-n-1<p\leq 0$, $\tilde{h}$ is constant. Now, the equation $\cK=h^{1-p}$ implies that $h$ is also constant.
\end{proof}

It might be of independent interest that \autoref{s2:lem-isotropic} is, in fact, a simple consequence of the following two general identities.
\begin{lemma}
	Let $K$ be a smooth, strictly convex body. Then
	\eq{
	\int \bar{\nabla}\log \frac{h^{n+2}}{\cK}\otimes x dV&=0,\\
	\int X\otimes \frac{x}{h} dV&=\left(\frac{1}{n+1}\int dV\right) \operatorname{Id}.
	}
\end{lemma}
\pf{
For the background in centro-affine geometry, see \cite{Mil23}. By \cite[Thm 1.3]{HI24}, we have
\eq{
\Delta X+nX=h\bar{\nabla}\log \frac{h^{n+2}}{\cK},
}
where $\Delta$ is defined in \cite[p. 3]{HI24}.
Moreover, by the centro-affine Gauss equation for $\xi^{\ast}(x):=x/h(x)$, we have
$\Delta\xi^{\ast}+n\xi^{\ast}=0$. Hence, the first identity follows by integrating by parts (cf. \cite[(4.9)]{Mil23}). The second identity follows from the divergence theorem; see the proof of \eqref{div x x polar}.
}
Now \autoref{s2:lem-isotropic} follows from
\eq{
	\int x\otimes x dV-\left(\frac{1}{n+1}\int dV\right) \operatorname{Id}&=-\int \bar{\nabla}\log h\otimes x dV\\
	&=-\frac{1}{n+1+p}\int  \bar{\nabla}\log \frac{h^{n+2}}{\cK}\otimes x dV\\
	&=0.
}

\section*{Acknowledgment}
We would like to thank the referees for their comments.
The work of the first author was supported by the National Key Research and Development Program of China 2021YFA1001800, the National Natural Science Foundation of China 12101027, and the Fundamental Research Funds for the Central Universities. Both authors were supported by the Austrian Science Fund (FWF): Project P36545.

\providecommand{\bysame}{\leavevmode\hbox to3em{\hrulefill}\thinspace}

	\vspace{2mm}
	\textsc{School of Mathematical Sciences, Beihang University, Beijing 100191, China}
	\email{\href{mailto:huyingxiang@buaa.edu.cn}{huyingxiang@buaa.edu.cn}}
	
	\vspace{3mm}
\textsc{Institut f\"{u}r Diskrete Mathematik und Geometrie,\\ Technische Universit\"{a}t Wien, Wiedner Hauptstra{\ss}e 8-10,\\ 1040 Wien, Austria,} \email{\href{mailto:mohammad.ivaki@tuwien.ac.at}{mohammad.ivaki@tuwien.ac.at}}
	
\end{document}